\documentclass{amsart}

\newtheorem{theorem}[equation]{Theorem}
\newtheorem{lemma}[equation]{Lemma}

\newtheorem{proposition}[equation]{Proposition}

\numberwithin{equation}{section}

\usepackage{amscd,amsmath,amssymb,verbatim}

\begin{document}

\title[Hasse invariants]{Hasse invariants and mod $p$ solutions of $A$-hypergeometric systems} 
\author{Alan Adolphson}
\address{Department of Mathematics\\
Oklahoma State University\\
Stillwater, Oklahoma 74078}
\email{adolphs@math.okstate.edu}
\author{Steven Sperber}
\address{School of Mathematics\\
University of Minnesota\\
Minneapolis, Minnesota 55455}
\email{sperber@math.umn.edu}
\date{\today}
\begin{abstract}
Igusa noted that the Hasse invariant of the Legendre family of elliptic curves over a finite field of odd characteristic is a solution mod $p$ of a Gaussian hypergeometric equation.  We show that any family of exponential sums over a finite field has a Hasse invariant which is a sum of products of mod $p$ solutions of $A$-hypergeometric systems.
\end{abstract}
\maketitle


\section{Introduction}

Igusa\cite{I} noted that when $p$ is an odd prime, the Hasse invariant 
\[ H(\lambda) = (-1)^{(p-1)/2} \sum_{i=0}^{(p-1)/2} \binom{(p-1)/2}{i}^2 \lambda^i \]
of the Legendre family of elliptic curves $y^2 = x(x-1)(x-\lambda)$ over a finite field of characteristic $p$ is a mod $p$ solution of the Gaussian hypergeometric equation
\[ \lambda(1-\lambda)y'' + (1-2\lambda)y' - (1/4)y= 0. \]
In fact, $(-1)^{(p-1)/2}H(\lambda)$ is congruent mod $p$ to the truncation of the Gaussian hypergeometric series
${}_2F_1(1/2,1/2,1;\lambda)$ at $\lambda^{(p-1)/2}$.
The numerator of the zeta function of the elliptic curve has a unit root if and only if $H(\lambda)\neq 0$.  More precisely, let 
\[ Z_\lambda(T) = \frac{(1-\alpha_1(\lambda)T)(1-\alpha_2(\lambda)T)}{(1-T)(1-qT)} \]
be the zeta function of the (projectivized) elliptic curve over ${\mathbb F}_q$, $q=p^a$, when $\lambda\neq 0,1,\infty$, where $\alpha_1(\lambda)$ denotes the unit root when $\lambda$ is not supersingular.  Dwork\cite[Equation~(6.29)]{D1} gave a formula for $\alpha_1(\lambda)$ in terms of values of (analytic continuations of) $p$-adic hypergeometric functions.  The reduction mod $p$ of Dwork's formula is
\[ \alpha_1(\lambda) \equiv H(\lambda)H(\lambda^p)\cdots H(\lambda^{p^{a-1}}) \pmod{p}. \]
The purpose of this article is to give a generalization of this congruence to arbitrary families of exponential sums.  In the process we extend earlier work of Beukers\cite{B} describing mod $p$ solutions of $A$-hypergeometric systems.

Let $p$ be a prime, let ${\mathbb F}_q$ be the finite field of $q=p^a$ elements, and let
\[ f_{\lambda} = \sum_{j=1}^N \lambda_jx^{{\bf a}_j}\in{\mathbb F}_q[x_1^{\pm 1},\dots,x_n^{\pm 1}]. \]
Put $A=\{{\bf a}_1,\dots,{\bf a}_N\}\subseteq{\mathbb Z}^n$.  We make no assumptions on $A$ until Section 7.
Let $\Psi:{\mathbb F}_q\to{\mathbb Q}_p(\zeta_p)^\times$ be a nontrivial additive character and let $\omega:{\mathbb F}_q^\times\to{\mathbb Q}_p(\zeta_{q-1})^\times$ be the Teichm\"uller character.  We consider the exponential sum
\begin{equation}
S(f_\lambda,{\bf e}) = \sum_{x=(x_1,\dots,x_n)\in({\mathbb F}_q^\times)^n} \omega(x_1)^{-e_1}\cdots\omega(x_n)^{-e_n} \Psi(f_\lambda(x))\in {\mathbb Q}_p(\zeta_p,\zeta_{q-1}),
\end{equation}
where we set ${\bf e} = (e_1,\dots,e_n)\in{\mathbb Z}^n$.   

It can happen that $S(f_\lambda,{\bf e})=0$ for all $\lambda=(\lambda_1,\dots,\lambda_N)\in{\mathbb F}_q^N$ (see Eq.~(6.2) below).
If $S(f_\lambda,{\bf e})\neq 0$ for some $\lambda$, then using the method of Ax\cite{A} and Stickelberger's Theorem it is not hard to show (see Eq.~(6.3) below) that there is a nonnegative integer $C$ and a polynomial $\hat{H}\in({\mathbb Q}\cap{\mathbb Z}_p)[\Lambda_1,\dots,\Lambda_N]$ of degree $\leq q-1$ in each variable and nonzero modulo~$p$, both depending only on $A$ and ${\bf e}$, such that
\begin{equation}
S(f_\lambda,{\bf e}) \equiv \hat{H}(\omega(\lambda_1),\dots,\omega(\lambda_N))\pi^C\pmod{\pi^{C+1}}, 
\end{equation}
where $\pi$ is a certain uniformizer for the field ${\mathbb Q}_p(\zeta_p,\zeta_{q-1})$.  Let $H(\lambda_1,\dots,\lambda_N)\in{\mathbb F}_p[\lambda_1,\dots,\lambda_N]$ be the reduction mod $p$ of $\hat{H}(\omega(\lambda_1),\dots,\omega(\lambda_N))$.  We call $H(\lambda)$ the {\it Hasse invariant\/} of the family of exponential sums $S(f_\lambda,{\bf e})$.  One thus has 
\begin{equation}
{\rm ord}_p\;S(f_\lambda,{\bf e}) \geq  \frac{C}{p-1}
\end{equation}
with equality holding if and only if $H(\lambda_1,\dots,\lambda_N)\neq 0$, where ${\rm ord}_p$ denotes the $p$-adic valuation normalized by ${\rm ord}_p\;p = 1$.

The estimate (1.3) improves on our previous work (\cite[Theorem 2]{AS1}, \cite[Theorem~3.3]{AS2}), where the lower bound was expressed in terms of the Newton polyhedron of $f_\lambda$.  
In the case where all multiplicative characters are trivial and $f_\lambda$ is an ordinary polynomial, estimate (1.3) has already been observed by Moreno, Shum, Castro, and Kumar\cite[Theorems~7 and~9]{M}, who analyzed $S(f_\lambda,{\bf 0})$ using the notion of $p$-degree.  These sums over ${\mathbb A}^n$ have been studied further by R. Blache\cite{Bl1,Bl2,Bl3}.  The notion of $p$-degree also plays a key role in the proof of our results.

The main point of this article is to show that $H(\lambda)$ is a sum of mod~$p$ solutions of certain $A$-hypergeometric systems.  We first show (Sections~2 and~3) that associated to certain lattice points $\gamma\in{\mathbb Z}^n$ are polynomials $F_\gamma(\lambda)\in{\mathbb F}_p[\lambda_1,\dots,\lambda_N]$ that satisfy $A$-hypergeometric systems mod $p$.  This generalizes earlier work of Beukers\cite{B}.  We then show (Section~6) that associated to $S(f_\lambda,{\bf e})$ is a finite set $\Gamma$ of sequences $(\gamma_0,\gamma_1,\dots,\gamma_{a-1})$ of these lattice points such that
\[ H(\lambda_1,\dots,\lambda_N) = \sum_{(\gamma_0,\dots,\gamma_{a-1})\in \Gamma} F_{\gamma_0}(\lambda) F_{\gamma_1}(\lambda^p)\cdots F_{\gamma_{a-1}}(\lambda^{p^{a-1}}) \]
(see Theorem 6.6 below for the precise statement).  The proof provides an explicit calculation of $C$ and the $\gamma_i$ in terms of the set $A$ and the vector~${\bf e}$ (see Section~6).  Using the natural toric decomposition of affine space, the result extends to exponential sums on ${\mathbb A}^n$ as well (see Theorem 7.9).  Although not needed in the rest of the article, we describe in Section 4 some relations between the $F_\gamma$ and truncations of series solutions of $A$-hypergeometric systems in characteristic~$0$.  

\section{$A$-hypergeometric systems}

Let ${\mathbb N}A$ (where ${\mathbb N}=\{0,1,\dots\}$) denote the semigroup generated by $A$ and let ${\mathbb Z}A\subseteq{\mathbb Z}^n$ denote the group generated by $A$.  Let $L\subseteq{\mathbb Z}^N$ be the lattice of relations on $A$:
\[ L = \bigg\{l=(l_1,\dots,l_N)\in{\mathbb Z}^N\mid \sum_{i=1}^N l_i{\bf a}_i = {\bf 0}\bigg\}. \]
Let $\beta = (\beta_1,\dots,\beta_n)\in{\mathbb C}^n$.  The $A$-{\it hypergeometric system with parameter $\beta$\/} is the system of partial differential operators in variables $\lambda_1,\dots,\lambda_N$ consisting of the {\it box operators}
\begin{equation}
\Box_l = \prod_{l_i>0}\bigg(\frac{\partial}{\partial \lambda_i}\bigg)^{l_i} - \prod_{l_i<0}\bigg(\frac{\partial}{\partial \lambda_i}\bigg)^{-l_i} \quad \text{for $l\in L$} 
\end{equation}
and the {\it Euler\/} (or {\it homogeneity\/}) {\it operators}
\begin{equation}
Z_i = \sum_{j=1}^N a_{ji}\lambda_j\frac{\partial}{\partial \lambda_j} - \beta_i\quad\text{for $i=1,\dots,n$,} 
\end{equation}
where ${\bf a}_j=(a_{j1},\dots,a_{jn})$.  If there is a linear form $h$ on ${\mathbb R}^n$ such that $h({\bf a}_i) = 1$ for $i=1,\dots,N$, we call this system {\it nonconfluent\/}; otherwise, we call it {\it confluent}.

If the $\beta_i$ are $p$-integral rational numbers, one can consider (2.1) and (2.2) modulo~$p$ and ask for their solutions in ${\mathbb F}_p((\lambda_1,\dots,\lambda_N))$, the quotient field of the formal power series ring ${\mathbb F}_p[[\lambda_1,\dots,\lambda_N]]$.  However, this is not the relevant system for studying exponential sums: one needs the reduction mod $p$ of the $p$-adically normalized $A$-hypergeometric system.   In \cite{D}, Dwork normalized the system corresponding to the Bessel differential equation in order to describe the variation of cohomology of the family of Kloosterman sums over a finite field.  Essentially, this normalization guarantees that the $p$-adic radius of convergence of series solutions is equal to 1.  For a more recent example of this see \cite{AS}, where we extend one of Dwork's results to arbitrary families of exponential sums.  The normalization involves simply replacing each $\lambda_i$ by $\pi\lambda_i$, where $\pi^{p-1} = -p$ (note that $\pi$ is a uniformizer for the field ${\mathbb Q}_p(\zeta_p,\zeta_{q-1})$).  This change of variable has no effect on the Euler operators, however, the box operator $\Box_l$ is transformed to 
\[ \pi^{-\sum_{l_i>0} l_i}\prod_{l_i>0}\bigg(\frac{\partial}{\partial \lambda_i}\bigg)^{l_i} - \pi^{\sum_{l_i<0} l_i}\prod_{l_i<0} \bigg(\frac{\partial}{\partial \lambda_i}\bigg)^{-l_i}. \]
We then multiply by the smallest power of $\pi$ that makes both coefficients $p$-integral:
\[ \pi^{\max\{\sum_{l_i>0} l_i,-\sum_{l_i<0} l_i\}}\bigg(\pi^{-\sum_{l_i>0} l_i}\prod_{l_i>0}\bigg(\frac{\partial}{\partial \lambda_i}\bigg)^{l_i} - \pi^{\sum_{l_i<0} l_i}\prod_{l_i<0} \bigg(\frac{\partial}{\partial \lambda_i}\bigg)^{-l_i}\bigg). \]
Reducing this expression mod $\pi$ gives for $l\in L$ the operator
\begin{equation}
\overline{\Box}_l = \begin{cases} \prod_{l_i>0} (\partial/\partial \lambda_i)^{l_i} & \text{if $\sum_{i=1}^N l_i > 0$,} \\ \prod_{l_i<0}(\partial/\partial\lambda_i)^{-l_i} & \text{if $\sum_{i=1}^N l_i<0$,} \\
\prod_{l_i>0}(\partial/\partial\lambda_i)^{l_i} - \prod_{l_i<0}(\partial/\partial\lambda_i)^{-l_i} & \text{if $\sum_{i=1}^N l_i = 0$.} \end{cases}
\end{equation}

When $\beta$ is an $n$-tuple of $p$-integral rational numbers, one can thus consider two $A$-hyperge\-o\-met\-ric systems modulo $p$: the first consisting of (2.1) and (2.2) and the second consisting of (2.2) and (2.3).  It is the second system whose mod $p$ solutions are related to the exponential sums~(1.1).  However, note that when the system is nonconfluent it is the third possibility in (2.3) that always holds and the two systems mod $p$ are identical.   

Let $\beta\in{\mathbb Z}^n$ and set 
\begin{align*}
 U(\beta) &= \bigg\{ u=(u_1,\dots,u_N)\in{\mathbb Z}^N\mid \sum_{i=1}^N u_i{\bf a}_i =  \beta \bigg\}, \\
U^+(\beta) &= \bigg\{ u=(u_1,\dots,u_N)\in{\mathbb N}^N\mid \sum_{i=1}^N u_i{\bf a}_i =  \beta \bigg\}.
\end{align*}
It is clear that $U(\beta)\neq\emptyset$ (resp.: $U^+(\beta)\neq\emptyset$) if and only if $\beta\in{\mathbb Z}A$ (resp.: $\beta\in{\mathbb N}A$).  For any nonnegative integer $k$, put
\[ U^+_k(\beta) = \{u\in U^+(\beta)\mid u_i\leq k\text{ for $i=1,\dots,N$}\}. \]
We define the {\it weight\/} of $\beta$, $w(\beta)$, to be
\[ w(\beta) = \min\bigg\{\sum_{i=1}^N u_i\mid u\in U^+(\beta)\bigg\}, \]
and we say that $u\in U^+(\beta)$ is {\it minimal\/} if $\sum_{i=1}^N u_i = w(\beta)$.  Let $U^+_{\rm min}(\beta)$ be the subset of minimal elements of $U^+(\beta)$.
We say that $\beta$ is {\it good\/} if $U^+_{\rm min}(\beta)\subseteq U^+_{p-1}(\beta)$ and we say that $\beta$ is {\it very good\/} if $U^+(\beta) = U^+_{p-1}(\beta)$.  When $\beta$ is good (resp.: very good), we can define
\begin{equation}
F_\beta(\lambda_1,\dots,\lambda_N) = \sum_{u\in U^+_{\rm min}(\beta)} \frac{\lambda_1^{u_1}\cdots\lambda_N^{u_N}} {u_1!\cdots u_N!}\in{\mathbb F}_p[\lambda_1,\dots,\lambda_N] \end{equation}
(resp.:
\begin{equation}
G_\beta(\lambda_1,\dots,\lambda_N) = \sum_{u\in U^+(\beta)} \frac{\lambda_1^{u_1}\cdots\lambda_N^{u_N}} {u_1!\cdots u_N!}\in{\mathbb F}_p[\lambda_1,\dots,\lambda_N] ).
\end{equation}
For explicit examples of such polynomials, we refer to Example 1 in Section 3 and Example 2 in Section 5.

Let
\begin{align*}
\sigma_A(\beta) &= \{\gamma\in (\beta + p{\mathbb Z}^n)\cap {\mathbb N}A\mid \text{$\gamma$ is good}\}, \\
\tau_A(\beta) &= \{\gamma\in (\beta + p{\mathbb Z}^n)\cap {\mathbb N}A\mid \text{$\gamma$ is very good}\}. 
\end{align*}
Both $\sigma_A(\beta)$ and $\tau_A(\beta)$ are finite sets as both are contained in the finite set 
\[ \bigg\{ \sum_{i=1}^N c_i{\bf a}_i\mid c_i\in\{0,1,\dots,p-1\}\text{ for $i=1,\dots,N$}\bigg\}. \]
Even if $(\beta + p{\mathbb Z}^n)\cap{\mathbb N}A$ is nonempty, the set $\tau_A(\beta)$ may be empty; however, 
\[ (\beta+p{\mathbb Z}^n)\cap{\mathbb N}A\neq\emptyset\quad \text{implies}\quad \sigma_A(\beta)\neq\emptyset. \]
More precisely, we have the following result.

\begin{lemma}
Let $\gamma\in(\beta + p{\mathbb Z}^n)\cap{\mathbb N}A$ be such that $w(\gamma)\leq w(\gamma')$ for all $\gamma'\in(\beta + p{\mathbb Z}^n)\cap{\mathbb N}A$.  Then $\gamma$ is good.
\end{lemma}

\begin{proof}
Let $\gamma\in(\beta+p{\mathbb Z}^n)\cap{\mathbb N}A$ satisfy the hypothesis of the lemma.  If $u\in U^+_{\rm min}(\gamma)$ had $u_i\geq p$ for some $i$, define $u' = (u_1',\dots,u_N')$ by
\[ u'_i = \begin{cases} u_i & \text{if $u_i\leq p-1$,} \\ u_i-p & \text{if $u_i\geq p$,} \end{cases} \]
and set $\gamma' = \sum_{i=1}^N u'_i{\bf a}_i$.  Then $\gamma'\in(\beta + p{\mathbb Z}^n)\cap{\mathbb N}A$ but
\[ w(\gamma') \leq \sum_{i=1}^N u'_i<\sum_{i=1}^N u_i = w(\gamma), \]
contradicting the choice of $\gamma$.
\end{proof}

Let $S_A(\beta)\subseteq {\mathbb F}_p((\lambda_1,\dots,\lambda_N))$ (resp.: $T_A(\beta)$) be the solution space of the mod $p$ $A$-hypergeometric system (2.2), (2.3) (resp.: (2.1), (2.2)).  Since differential operators in characteristic $p$ annihilate all $p$-th powers, we consider $S_A(\beta)$ and $T_A(\beta)$ as vector spaces over the field ${\mathbb F}_p((\lambda_1^p,\dots,\lambda_N^p))$.

The following result is a slight generalization of Beukers\cite[Proposition 4.1]{B}.  Although it is not needed for this article, we include it to provide some context for the results of the next section.

\begin{theorem}
The set of polynomials $\{G_\gamma\mid \gamma\in\tau_A(\beta)\}$ is a basis for $T_A(\beta)$.
\end{theorem}

\begin{proof}[Sketch of proof]
One can repeat the argument of \cite{B}.  Although Beukers makes an overriding assumption that his system is nonconfluent, that hypothesis is not used in the part of his paper dealing with mod $p$ solutions.  He also assumes that ${\mathbb N}A$ equals the set of all lattice points in the real cone generated by $A$, a condition that is needed in the proof of \cite[Proposition 4.1]{B}.  Using our set of solutions $\{G_\gamma\}_{\gamma\in\tau_A(\beta)}$ in place of his, the proof of \cite[Proposition 4.1]{B} remains valid in the unsaturated case.  (In the saturated case, our set of solutions is identical to the one constructed by Beukers.)
\end{proof}

\section{Mod $p$ solutions}

The main purpose of this section is to prove the following result.

\begin{theorem}
The set of polynomials $\{F_\gamma\mid\gamma\in\sigma_A(\beta)\}$ is a basis for $S_A(\beta)$.  
\end{theorem}

\begin{proof}
We first show that each $F_\gamma$ is a solution of the system (2.2), (2.3).
The definition of $\sigma_A(\beta)$ implies that $F_\gamma$ satisfies Equations~(2.2) modulo $p$, so it remains to show that $F_\gamma$ is a mod $p$ solution of the box operators (2.3).  Fix $l=(l_1,\dots,l_N)$ satisfying $\sum_{i=1}^N l_i{\bf a}_i = {\bf 0}$.  First suppose that $\sum_{i=1}^N l_i>0$.  We must show that
\begin{equation}
\prod_{l_i>0} \bigg(\frac{\partial}{\partial \lambda_i}\bigg)^{l_i}(\lambda_1^{u_1}\cdots\lambda_N^{u_N}) = 0
\quad\text{ for all $u\in U^+_{\rm min}(\gamma)$.}
\end{equation}
Fix $u\in U^+_{\rm min}(\gamma)$ and consider $v\in{\mathbb Z}^N$ defined by $v_i= u_i- l_i$ for $i=1,\dots,N$.  Then we have
\[ \sum_{i=1}^N v_i{\bf a}_i = \sum_{i=1}^N u_i{\bf a}_i - \sum_{i=1}^N l_i{\bf a}_i = \gamma \]
and 
\[ \sum_{i=1}^N v_i = \sum_{i=1}^N u_i - \sum_{i=1}^N l_i < \sum_{i=1}^N u_i = w(\gamma). \]
If $v_i\geq 0$ for all $i$, we would have $v\in U^+(\gamma)$ with $\sum_{i=1}^N v_i<\sum_{i=1}^N u_i$, contradicting $u\in U^+_{\min}(\gamma)$.  It follows that $v_i<0$ for some $i$.  But $u_i\geq 0$ for all $i$ implies $u_i- l_i\geq 0$ if $l_i\leq 0$, so we must have $u_i-l_i<0$ for some $l_i>0$.  This immediately implies~(3.2).

The proof is similar if $\sum_{i=1}^N l_i<0$ so suppose that $\sum_{i=1}^N l_i = 0$.  We must show that
\begin{equation}
\bigg(\prod_{l_i>0} \bigg(\frac{\partial}{\partial \lambda_i}\bigg)^{l_i} - \prod_{l_i<0} \bigg(\frac{\partial}{\partial \lambda_i}\bigg)^{-l_i}\bigg)\bigg(\sum_{u\in U^+_{\rm min}(\gamma)} \frac{\lambda_1^{u_1}\cdots\lambda_N^{u_N}}{u_1!\cdots u_N!}\bigg) = 0.
\end{equation} 
Write $l = l_+-l_-$, where $l_+ = (l^+_1,\dots,l^+_N)$ and $l_- = (l^-_1,\dots,l^-_N)$ are defined by
\[ l^+_i = \max\{l_i,0\}, \quad l^-_i = \max\{-l_i,0\}. \]
Let 
\begin{align*} 
U_1(\gamma) &= \{u\in U^+_{\rm min}(\gamma)\mid u_i-l_i\geq 0\text{ for $i=1,\dots,N$}\}, \\
U_2(\gamma) &= \{v\in U^+_{\rm min}(\gamma)\mid v_i + l_i\geq 0\text{ for $i=1,\dots,N$}\}.
\end{align*}
Then we have
\begin{align}
\prod_{l_i>0}\bigg(\frac{\partial}{\partial \lambda_i}\bigg)^{l_i}(F_\gamma) &= \sum_{u\in U_1(\gamma)} \frac{\lambda_1^{u_1-l^+_1}\cdots\lambda_N^{u_N-l^+_N}}{(u_1-l^+_1)!\cdots(u_N-l^+_N)!}, \\
\prod_{l_i<0}\bigg(\frac{\partial}{\partial \lambda_i}\bigg)^{-l_i}(F_\gamma) &= \sum_{v\in U_2(\gamma)} \frac{\lambda_1^{v_1-l^-_1}\cdots\lambda_N^{v_N-l^-_N}}{(v_1-l^-_1)!\cdots(v_N-l^-_N)!}.
\end{align}
Note that we have a one-to-one correspondence between the sets $U_1(\gamma)$ and $U_2(\gamma)$:  if $u\in U_1(\gamma)$, then $v=u-l\in U_2(\gamma)$.  The inverse of this map is given by: send $v\in U_2(\gamma)$ to $v+l\in U_1(\gamma)$.  Furthermore, if $u\in U_1(\gamma)$ and $v\in U_2(\gamma)$ are related under this correspondence, then
\[ u_i- l_i^+ = v_i-l^-_i\quad\text{for $i=1,\dots,N$}. \]
This shows that the right-hand sides of (3.4) and (3.5) are equal, and (3.3) follows immediately.  We have established that $F_\gamma$ is a solution of the system (2.2), (2.3).

The functions $\{F_\gamma\}_{\gamma\in\sigma_A(\beta)}$ are polynomials in $\lambda_1,\dots,\lambda_N$ of degree $\leq p-1$ in each variable and the sets $\{U^+_{\rm min}(\gamma)\}_{\gamma\in\sigma_A(\beta)}$ are mutually disjoint, hence the set $\{F_\gamma\}_{\gamma\in\sigma_A(\beta)}$ is linearly independent over ${\mathbb F}_p((\lambda_1^p,\dots,\lambda_N^p))$.  It remains to show that every solution of (2.2), (2.3) is an ${\mathbb F}_p((\lambda_1^p,\dots,\lambda_N^p))$-linear combination of these polynomials.

Let $G$ be any solution of (2.2), (2.3) in ${\mathbb F}_p((\lambda_1,\dots,\lambda_N))$.  Then $G$ is a (possibly infinite) sum of expressions of the form $\lambda_1^{pk_1}\cdots\lambda_N^{pk_N}G_{k_1\dots k_N}$ ($k_1,\dots,k_N\in{\mathbb Z}$), where $G_{k_1\dots k_N}$ is a polynomial of degree $\leq p-1$ in each variable $\lambda_1,\dots,\lambda_N$.  Furthermore,  $G$ is a solution of (2.2), (2.3) if and only if each $G_{k_1\dots k_N}$ is.  So we may take $G$ to be a polynomial of degree $\leq p-1$ in each variable.  To satisfy~(2.2), each monomial $\lambda_1^{k_1}\cdots \lambda_N^{k_N}$ in $G$ must satisfy
\[ \sum_{i=1}^N k_i{\bf a}_i \equiv\beta\pmod{p{\mathbb Z}^n}. \]
It follows that we may write
\[ G = \sum_{j=1}^J \sum_{u\in U^+_{p-1}(\gamma_j)} c_u\lambda_1^{u_1}\cdots\lambda_N^{u_N}, \]
where $\gamma_1,\dots,\gamma_J\equiv\beta\pmod{p{\mathbb Z}^n}$. 
Put
\[ G_{\gamma_j} = \sum_{u\in U^+_{p-1}(\gamma_j)}c_u\lambda_1^{u_1}\cdots\lambda_N^{u_N}. \]
Each $G_{\gamma_j}$ satisfies (2.2).  Suppose $l\in L$ and $\sum_{i=1}^N l_i>0$.  The corresponding box operator, 
\[ \overline{\Box}_l = \prod_{l_i>0} \bigg(\frac{\partial}{\partial \lambda_i}\biggr)^{l_i}, \]
annihilates $G$ if and only if it annihilates each monomial in $G$; so it annihilates $G$ if and only if it annihilates each $G_{\gamma_j}$.  A similar argument applies when $\sum_{i=1}^N l_i<0$, so assume $\sum_{i=1}^N l_i = 0$.  The corresponding box operator is 
\[ \overline{\Box}_l = \prod_{l_i>0}\bigg(\frac{\partial}{\partial\lambda_i}\bigg)^{l_i} - \prod_{l_i<0} \bigg(\frac{\partial}{\partial\lambda}_i\bigg)^{-l_i}. \]
It is straightforward to check that the monomials appearing in the $\overline{\Box}_l(G_{\gamma_j})$ form disjoint sets for $j=1,\dots,J$.  It follows that $G$ satisfies (2.3) if and only if each $G_{\gamma_j}$ does.

We are finally reduced to the case where
\[ G = \sum_{u\in U^+_{p-1}(\gamma)} c_u\lambda_1^{u_1}\cdots\lambda_N^{u_N} \]
for some $\gamma\equiv\beta\pmod{p{\mathbb Z}^n}$.  We claim that $c_u = 0$ if $u\not\in U^+_{\rm min}(\gamma)$.  Pick $v\in U^+_{\rm min}(\gamma)$ and take $l=u-v\in L$.  If $u\not\in U^+_{\rm min}(\gamma)$, then $\sum_{i=1}^N u_i-v_i>0$, so the corresponding box operator is
\[ \overline{\Box}_l = \prod_{u_i>v_i} \bigg(\frac{\partial}{\partial\lambda_i}\bigg)^{u_i-v_i}. \]
Since $u_i-v_i\leq u_i$ for all $i$, it is clear that $\overline{\Box}_l(c_u\lambda_1^{u_1}\cdots\lambda_N^{u_N})\neq 0$ unless $c_u = 0$.

We now have
\begin{equation}
G = \sum_{u\in U^+_{p-1}(\gamma)\cap U^+_{\rm min}(\gamma)} c_u \lambda_1^{u_1}\cdots\lambda_N^{u_N}, 
\end{equation}
and we need to show that $G=0$ unless $\gamma$ is good (i.e., $U^+_{\rm min}(\gamma)\subseteq U^+_{p-1}(\gamma)$), and that when $\gamma$ is good, $G$ is a scalar multiple of the function $F_\gamma$ defined earlier.  If $U^+_{p-1}(\gamma)\cap U^+_{\rm min}(\gamma)=\emptyset$ then $G=0$.  If $U^+_{\rm min}(\gamma)$ is a singleton, then either $U^+_{\rm min}(\gamma)\subseteq U^+_{p-1}(\gamma)$ (so $\gamma$ is good and $G$ is clearly a scalar multiple of $F_\gamma$) or $U^+_{p-1}(\gamma)\cap U^+_{\rm min}(\gamma)=\emptyset$ (so $G=0$).  

So suppose $U^+_{\rm min}(\gamma)$ has at least two elements and $U^+_{p-1}(\gamma)\cap U^+_{\rm min}(\gamma)\neq\emptyset$.  Let $u\in U^+_{p-1}(\gamma)\cap U^+_{\rm min}(\gamma)$ and let $v\in U^+_{\rm min}(\gamma)$ (for notational convenience, we set $c_v = 0$ if $v\not\in U^+_{p-1}(\gamma)$).  Put $l=u-v\in L$.  Since $\sum_{i=1}^N u_i-v_i=0$, the corresponding box operator is
\[ \overline{\Box}_l = \prod_{u_i>v_i}\bigg(\frac{\partial}{\partial\lambda_i}\bigg)^{u_i-v_i} - \prod_{u_i<v_i} \bigg(\frac{\partial}{\partial\lambda}_i\bigg)^{v_i-u_i}. \]
The coefficient of $\prod_{i=1}^N \lambda_i^{\min\{u_i,v_i\}}$ in $\overline{\Box}_l(G)$ is
\begin{equation}
c_u\prod_{u_i>v_i}u_i(u_i-1)\cdots(v_i+1) - c_v\prod_{u_i<v_i}v_i(v_i-1)\cdots(u_i+1). \end{equation}
Since $u_i\leq p-1$ for $i=1,\dots,N$, the coefficient of $c_u$ in this expression is $\neq 0$.  If $\gamma$ is not good, then there exists $v\in U^+_{\rm min}(\gamma)$ such that $v\not\in U^+_{p-1}(\gamma)$.  Since $c_v=0$ and since $\overline{\Box}_l(G)=0$ implies the vanishing of (3.7), it follows that $c_u=0$.  Since $u$ was an arbitrary element of $U^+_{p-1}(\gamma)\cap U^+_{\rm min}(\gamma)$, we conclude that $G=0$.  If $\gamma$ is good, then $U^+_{\rm min}(\gamma)\subseteq U^+_{p-1}(\gamma)$, so the sum in (3.6) is over $U^+_{\rm min}(\gamma)$.  If $u,v\in U^+_{\rm min}(\gamma)$, then the coefficients of $c_u$ and $c_v$ in (3.7) are both $\neq 0$.  Equation (3.7) and the vanishing of $\overline{\Box}_l(G)$ then imply that the value of $c_u$ for one $u\in U^+_{\rm min}(\gamma)$ determines the values of $c_v$ for all $v\in U^+_{\rm min}(\gamma)$.  This proves that the space of solutions of the form (3.6) is one-dimensional, hence $G$ is a scalar multiple of $F_\gamma$.
\end{proof}

{\bf Example 1:}  Let $A=\{{\bf a}_1,{\bf a}_2,{\bf a}_3,{\bf a}_4\}\subseteq{\mathbb Z}^3$, where ${\bf a}_1 = (1,0,0)$, ${\bf a}_2  = (0,1,0)$, ${\bf a}_3 = (0,0,1)$, and ${\bf a}_4 = (1,1,-1)$.  Then $(-1/2, -1/2, 0)$ is $p$-integral for any odd prime $p$ and taking $\beta = ((p-1)/2, (p-1)/2,0)\in{\mathbb Z}^3$ we have $\beta\equiv(-1/2,-1/2,0)\pmod{p}$.  The nonnegative integer solutions of the system of equations
\[ \sum_{i=1}^4 c_i{\bf a}_i = \bigg(\frac{p-1}{2},\frac{p-1}{2},0\bigg) \]
are given by
\[ c_1 = c_2 = \frac{p-1}{2}-l,\quad c_3=c_4 = l,\quad\text{for $l=0,1,\dots,\frac{p-1}{2}$.} \]
We thus have $U^+(\beta) = U^+_{p-1}(\beta) = U^+_{\min}(\beta)$, $w(\beta) = p-1$, and
\[ U^+_{\min}(\beta) = 
\bigg\{\bigg(\frac{p-1}{2}-l,\frac{p-1}{2}-l,l,l\bigg)\mid l\in\{0,1,\dots,(p-1)/2\}\bigg\}.  \]
In particular, $\beta$ is very good and
\[ F_\beta(\lambda) = \sum_{l=0}^{(p-1)/2} \frac{\lambda_1^{(p-1)/2-l}\lambda_2^{(p-1)/2-l}\lambda_3^l\lambda_4^l} 
{((p-1)/2-l)!^2\,l!^2}. \]
This can be simplified by multiplying by $(-1)^{(p+1)/2}((p-1)/2)!^2$ ($\equiv 1\pmod{p}$) to give
\[ F_\beta(\lambda) = -(-\lambda_1\lambda_2)^{(p-1)/2}\sum_{l=0}^{(p-1)/2}\binom{(p-1)/2}{l}^2\bigg(\frac{\lambda_3\lambda_4}{\lambda_1\lambda_2}\bigg)^l. \]

\section{Mod $p$ solutions and $A$-hypergeometric series}

The results of this section are not needed elsewhere in this paper, however, we include them with a view to future applications.

Fix $\beta\in{\mathbb Z}^n$ for which $(\beta+p{\mathbb Z}^n)\cap{\mathbb N}A\neq\emptyset$.  By Lemma~2.6 we may choose $\gamma\in
(\beta+p{\mathbb Z}^n)\cap{\mathbb N}A$ to be good.  Consider the associated solution
\[ F_\gamma(\lambda) = \sum_{u\in U^+_{\min}(\gamma)} \frac{\lambda_1^{u_1}\cdots\lambda_N^{u_N}}{u_1!\cdots u_N!} \]
of the mod $p$ system (2.2), (2.3).  We shall compare $F_\gamma$ with a formal solution of the $A$-hypergeometric system in characteristic~0 consisting of the operators
\begin{equation}
\sum_{j=1}^N a_{ji}\lambda_j\frac{\partial}{\partial\lambda_j} - \frac{\gamma}{1-p}\quad\text{for $i=1,\dots,n$}
\end{equation}
and
\begin{equation}
\pi^{-\sum_{l_i>0} l_i}\prod_{l_i>0}\bigg(\frac{\partial}{\partial\lambda_i}\bigg)^{l_i} - \pi^{\sum_{l_i<0}l_i}\prod_{l_i<0} \bigg(\frac{\partial}{\partial\lambda_i}\bigg)^{-l_i}\quad\text{for $l\in L$.}
\end{equation}

Fix an element $u^{(0)}\in U^+_{\min}(\gamma)$ and define $v^{(0)} = u^{(0)}/(1-p)$.  Then we have
\begin{equation}
\sum_{i=1}^N v^{(0)}_i{\bf a}_i = \frac{\gamma}{1-p}
\end{equation}
and, since $\gamma$ is good, we have
\[ -1\leq v^{(0)}_i\leq 0\quad\text{for $i=1,\dots,N$.} \]
Recall that the {\it negative support\/} of a vector $x=(x_1,\dots,x_N)\in{\mathbb C}^N$ is defined as ${\rm nsupp}(x) = \{i\in\{1,\dots,N\}\mid x_i\in{\mathbb Z}_{<0}\}$.
We make the hypothesis that $v^{(0)}$ has minimal negative support in the sense of \cite[Section~3.4]{SST}, i.e., that there is no element $l\in L$ such that ${\rm nsupp}(v^{(0)}+l)$ is a proper subset of ${\rm nsupp}(v^{(0)})$. Put
\[ N_{v^{(0)}} = \{l\in L\mid {\rm nsupp}(v^{(0)} + l) = {\rm nsupp}(v^{(0)})\}. \]
Since Eq.~(4.2) is obtained from the usual box operators by replacing each variable $\lambda_i$ by $\pi\lambda_i$, this implies (see \cite[Proposition~3.4.13]{SST}) that the formal series
\begin{equation}
\phi_{v^{(0)}}(\lambda) = \sum_{l\in N_{v^{(0)}}} \frac{[v^{(0)}]_{l_-}}{[v^{(0)}+l]_{l_+}} (\pi\lambda)^{v^{(0)}+l} 
\end{equation}
is a solution of the system (4.1), (4.2), where
\[ [v^{(0)}]_{l_-} = \prod_{l_i<0} \prod_{j=1}^{-l_i} (v^{(0)}_i-j+1) \]
and
\[ [v^{(0)}+l]_{l_+} = \prod_{l_i>0} \prod_{j=1}^{l_i} (v^{(0)}_i+j). \]

Note that 
\[ (\pi\lambda)^{-pv^{(0)}}\phi_{v^{(0)}}(\lambda) = \sum_{l\in N_{v^{(0)}}} \frac{[v^{(0)}]_{l_-}}{[v^{(0)}+l]_{l_+}} (\pi\lambda)^{u^{(0)}+l} \]
is a formal series where all powers of $\lambda$ lie in ${\mathbb Z}^N$.  Let $G_{v^{(0)}}(\lambda)\in{\mathbb Q}[\pi][\lambda]$  denote the truncation of $(\pi\lambda)^{-pv^{(0)}}\phi_{v^{(0)}}(\lambda)$ obtained by eliminating all terms except those containing monomials $\lambda_1^{j_1}\cdots\lambda_N^{j_N}$ satisfying $0\leq j_i\leq p-1$ for $i=1,\dots,N$.
More precisely, if we let 
\[ N_{v^{(0)}}' = \{l\in N_{v^{(0)}}\mid u^{(0)}+l\in U^+_{p-1}(\gamma)\}, \] 
then 
\begin{equation}
G_{v^{(0)}}(\lambda) = \sum_{l\in N_{v^{(0)}}'} \frac{[v^{(0)}]_{l_-}}{[v^{(0)}+l]_{l_+}} (\pi\lambda)^{u^{(0)}+l}.
\end{equation}

\begin{proposition}
If $v^{(0)}$ has minimal negative support, then $\pi^{-w(\gamma)}G_{v^{(0)}}(\lambda)$ has $p$-integral coefficients and
\[ \pi^{-w(\gamma)}G_{v^{(0)}}(\lambda)\equiv (u^{(0)}_1!)\cdots (u^{(0)}_N!)F_\gamma(\lambda)\pmod{\pi}. \]
\end{proposition}

\begin{proof}
Suppose that $u^{(0)}+l\in U^+(\gamma)$.  We claim that $l\in N_{v^{(0)}}$.  Suppose that for some $i$, $v^{(0)}_i+l_i\in{\mathbb Z}_{<0}$ but $v_i^{(0)}\not\in{\mathbb Z}_{<0}$.  Then we must have $v^{(0)}_i=0$ and $l_i<0$.  But $v^{(0)}_i=0$ implies $u^{(0)}_i=0$ so $u^{(0)}_i+l_i<0$, contradicting $u^{(0)}+l\in U^+(\gamma)$.  It follows that ${\rm nsupp}(v^{(0)}+l)\subseteq{\rm nsupp}(v^{(0)})$, and they must be equal since $v^{(0)}$ is assumed to have minimal negative support.  

It follows that $N'_{v^{(0)}} = \{l\in L\mid u^{(0)}+l\in U^+_{p-1}(\gamma)\}$.  We can thus write $G_{v^{(0)}}(\lambda)$ as the sum of two polynomials $G_1(\lambda)$ and $G_2(\lambda)$, where $G_1(\lambda)$ is the sum of those terms on the right-hand side of~(4.5) with $u^{(0)}+l\in U^+_{\min}(\gamma)$ and $G_2(\lambda)$ is the sum of those terms with $u^{(0)}+l\in U^+_{p-1}(\gamma)\setminus U^+_{\min}(\gamma)$.  We shall establish the proposition by showing that both $\pi^{-w(\gamma)}G_1(\lambda)$ and $\pi^{-w(\gamma)}G_2(\lambda)$ have $p$-integral coefficients and that
\begin{align}
\pi^{-w(\gamma)}G_2(\lambda)  &\equiv 0\pmod{\pi}, \\
\pi^{-w(\gamma)}G_1(\lambda) &\equiv (u^{(0)}_1!)\cdots (u^{(0)}_N!)F_\gamma(\lambda)\pmod{\pi}. 
\end{align}

Note that in the $p$-adic integers ${\mathbb Z}_p$ we have the equality
\begin{equation}
v^{(0)}_i = u^{(0)}_i + u^{(0)}_i p + u^{(0)}_i p^2 + \cdots. 
\end{equation}
If $l\in N'_{v^{(0)}}$ has $l_i\geq 0$, then $u^{(0)}_i+l_i\leq p-1$ implies $l_i\leq p-1-u^{(0)}_i$.  It then follows from~(4.9) that $\prod_{j=1}^{l_i}(v^{(0)}_i+j)$ is a $p$-adic unit, hence $[v^{(0)}+l]_{l_+}$ is a $p$-adic unit.
If $l\in N'_{v^{(0)}}$ has $l_i<0$, then $u^{(0)}_i+l_i\geq 0$ implies $l_i\geq -u_i^{(0)}$.  It follows from (4.9) that $\prod_{j=1}^{-l_i}(v^{(0)}_i-j+1)$ is a $p$-adic unit, hence $[v^{(0)}]_{l_-}$ is a $p$-adic unit.  This proves that the coefficient of $\lambda^{u^{(0)}+l}$ in $G_{v^{(0)}}(\lambda)$ is $p$-integral and divisible by $\pi^{\sum_{i=1}^N u^{(0)}_i+l_i}$.  If $u^{(0)}+l\not\in U^+_{\min}(\gamma)$, then $\sum_{i=1}^N u^{(0)}_i+l_i>w(\gamma)$, which establishes~(4.7).  To prove~(4.8), we must show that for $u^{(0)}+l\in U^+_{\min}(\gamma)$, 
\begin{equation}
\frac{[v^{(0)}]_{l_-}}{[v^{(0)}+l]_{l_+}}\equiv \frac{(u^{(0)}_1!)\cdots (u^{(0)}_N!)}{(u^{(0)}_1+l_1)!\cdots (u^{(0)}_N+l_N)!} \pmod{p}.
\end{equation}
If $l_i<0$, then by (4.9)
\begin{align*}
\prod_{j=1}^{-l_i} (v^{(0)}_i-j+1) &\equiv \prod_{j=1}^{-l_i} (u^{(0)}_i-j+1)\pmod{p} \\
 &= \frac{u^{(0)}_i!}{(u^{(0)}_i+l_i)!},
\end{align*}
and if $l_i>0$, then by (4.9)
\begin{align*}
\prod_{j=1}^{l_i} (v^{(0)}_i+j) &\equiv \prod_{j=1}^{l_i} (u^{(0)}_i+j)\pmod{p} \\
 &= \frac{(u^{(0)}_i+l_i)!}{u^{(0)}_i!}.
\end{align*}
These congruences imply (4.8).  Note that the coefficients of $\pi^{-w(\gamma)}G_1(\lambda)$ lie in ${\mathbb Q}\cap{\mathbb Z}_p$, so~(4.8) is actually a congruence mod $p$.
\end{proof}

{\bf Example 1(cont.):}  We maintain the notation of Example 1 from the previous section.  The lattice $L$ is given by $L=\{(-l,-l,l,l)\mid l\in{\mathbb Z}\}$.  Choose $u^{(0)} = ((p-1)/2,(p-1)/2,0,0)$, so that $v^{(0)} = (-1/2,-1/2,0,0)$.   Then $v^{(0)}$ has minimal negative support and the solution (4.4) of the system (4.1), (4.2) (with $\gamma/(1-p) = (-1/2, -1/2,0)$) is (using the Pochhammer notation $(x)_l = x(x+1)\cdots (x+l-1)$)
\[ \phi_{v^{(0)}}(\lambda) = \pi^{-1}\lambda_1^{-1/2}\lambda_2^{-1/2}\sum_{l=0}^\infty \bigg(\bigg(\frac{1}{2}\bigg)_{\!l}\bigg/l\,!\bigg)^2  \bigg(\frac{\lambda_3\lambda_4}{\lambda_1\lambda_2}\bigg)^l. \]
This gives
\[ (\pi\lambda)^{-pv^{(0)}}\phi_{v^{(0)}}(\lambda) = \pi^{p-1}\lambda_1^{(p-1)/2}\lambda_2^{(p-1)/2} \sum_{l=0}^\infty \bigg(\bigg(\frac{1}{2}\bigg)_{\!l}\bigg/l\,!\bigg)^2  \bigg(\frac{\lambda_3\lambda_4}{\lambda_1\lambda_2}\bigg)^l, \]
hence
\[ G_{v^{(0)}}(\lambda) = \pi^{p-1} (\lambda_1\lambda_2)^{(p-1)/2}\sum_{l=0}^{(p-1)/2}\bigg(\bigg(\frac{1}{2}\bigg)_{\!l}\bigg/l\,!\bigg)^2 \bigg(\frac{\lambda_3\lambda_4}{\lambda_1\lambda_2}\bigg)^l. \]
The assertion of Proposition 4.6 thus reduces to the congruence (see the formula for $F_\gamma(\lambda)$ in Example~1)
\begin{multline*}
 (\lambda_1\lambda_2)^{(p-1)/2}\sum_{l=0}^{(p-1)/2}\bigg(\bigg(\frac{1}{2}\bigg)_{\!l}\bigg/l\,!\bigg)^2 \bigg(\frac{\lambda_3\lambda_4}{\lambda_1\lambda_2}\bigg)^l\equiv \\
((p-1)/{2})!^2\sum_{l=0}^{(p-1)/2} \frac{\lambda_1^{(p-1)/2-l}\lambda_2^{(p-1)/2-l}\lambda_3^l\lambda_4^l} {(\frac{p-1}{2}-l)!^2\,l!^2}.
\end{multline*}

{\bf Remark:}  It would be interesting to know when the coefficients of the series~(4.4) are $p$-integral, as they are in this example (except for the factor $\pi^{-1}$).

\section{The $p$-weight of a set of lattice points}

Fix $q=p^a$.  In Section 3 we showed that good lattice points correspond to mod $p$ solutions of an $A$-hypergeometric system.  In this section we show that lattice points minimizing the ``$p$-weight'' of a set of lattice points satisfying condition (5.5) below give rise to sequences of length $a$ of good lattice points (Proposition~5.6 below).

Let $U^+_{q-1}=\{u=(u_1,\dots,u_N)\mid 0\leq u_i\leq q-1\text{ for $i=1,\dots,N$}\}$.  For each $u\in U^+_{q-1}$, we define $u^{(0)},\dots,u^{(a-1)}\in\{0,1,\dots,p-1\}^N$, $u^{(k)} = (u^{(k)}_1,\dots,u^{(k)}_N)$, by writing
\begin{equation}
u_i = u^{(0)}_i + u^{(1)}_i p +\cdots+ u^{(a-1)}_i p^{a-1} \quad\text{for $i=1,\dots,N$} 
\end{equation}
and we define
\begin{equation}
\gamma_k = \sum_{i=1}^N u^{(k)}_i{\bf a}_i\quad\text{for $k=0,\dots,a-1$.}
\end{equation}
Note that
\begin{equation}
\sum_{k=0}^{a-1} p^k\gamma_k = \sum_{i=1}^N u_i{\bf a}_i. 
\end{equation}
Define the $p$-{\it weight\/} of $u$, $w_p(u)$, to~be
\[ w_p(u) = \sum_{i=1}^N \sum_{k=0}^{a-1} u^{(k)}_i. \]
Note that since $w(\gamma_k)\leq\sum_{i=1}^N u^{(k)}_i$ with equality holding if and only if $u^{(k)}\in U^+_{\min}(\gamma_k)$ we have
\begin{equation}
w_p(u)\geq \sum_{k=0}^{a-1} w(\gamma_k),
\end{equation}
with equality holding if and only if $u^{(k)}\in U^+_{\min}(\gamma_k)$ for all $k$.

Let $M\subseteq{\mathbb Z}^n$ be a subset satisfying the condition:  For ${\bf a} = \sum_{i=1}^N c_i{\bf a}_i\in{\mathbb N}A$, 
\begin{equation}
\text{if ${\bf a}\in M$ and $c_{i_0}\geq q$ for some $i_0$, then ${\bf a}-(q-1){\bf a}_{i_0}\in M$}
\end{equation}
and let $U_M = \{u\in U^+_{q-1}\mid \sum_{i=1}^N u_i{\bf a}_i\in M\}$.  (In the next section we shall apply the results of this section taking $M={\bf e} + (q-1){\mathbb Z}^n$ for a certain choice of ${\bf e}\in{\mathbb Z}^n$.)  We assume for the remainder of this section that $U_M\neq\emptyset$ and we define the {\it $p$-weight\/} of $M$, $w_p(M)$, by
\[ w_p(M) = \min\{w_p(u)\mid u\in U_M\}.  \]
Put $U_{M,\min} = \{u\in U_M\mid w_p(u) = w_p(M)\}$.  The main result of this section is the following assertion.

\begin{proposition}
Let $u\in U_{M,\min}$ and let $\{u^{(k)}\}_{k=0}^{a-1}$ and $\{\gamma_k\}_{k=0}^{a-1}$ be defined by~$(5.1)$ and~$(5.2)$, respectively.  Then for $k=0,\dots,a-1$, $\gamma_k$ is good and $u^{(k)}\in U^+_{\min}(\gamma_k)$.
\end{proposition}

\begin{proof}
Let $u\in U_{M,{\min}}$ and suppose that $\gamma_{k_1}$ is not good.  Then there exists $v^{(k_1)}\in U^+_{\min}(\gamma_{k_1})$ with $v^{(k_1)}_{i_1}\geq p$ for some $i_1$.  For $k\neq k_1$ define $v^{(k)} = u^{(k)}$ and set $v=(v_1,\dots,v_N)$, where
\[ v_i = \sum_{k=0}^{a-1} v^{(k)}_ip^k\quad\text{for $i=1,\dots,N$}. \]
Note that
\begin{equation}
\sum_{i=1}^N v_i{\bf a}_i = \sum_{k=0}^{a-1} \bigg(\sum_{i=1}^N v^{(k)}_ip^k\bigg){\bf a}_i = \sum_{k=0}^{a-1} p^k\gamma_k = \sum_{i=1}^N u_i{\bf a}_i\in M
\end{equation}
and that (since $\sum_{i=1}^N v^{(k_1)}_i\leq \sum_{i=1}^N u^{(k_1)}_i$) 
\begin{equation}
\sum_{i=1}^N \sum_{k=0}^{a-1}v^{(k)}_i\leq w_p(u)=w_p(M).
\end{equation}

Suppose first that $k_1<a-1$.  Define 
\begin{equation}
\begin{cases} w^{(k)}_i = v^{(k)}_i & \text{if $k\neq k_1,k_1+1$ or if $i\neq i_1$,} \\
w^{(k_1)}_{i_1} = v^{(k_1)}_{i_1}-p, & \\
w^{(k_1+1)}_{i_1} = v^{(k_1+1)}_{i_1}+1, \end{cases}
\end{equation}
and put
\[ w_i = \sum_{k=0}^{a-1} w^{(k)}_i p^k\quad\text{for $i=1,\dots,N$.} \]
With this definition we have $w_i=v_i$ for $i=1,\dots,N$, so (5.7) implies
\[ \sum_{i=1}^N w_i{\bf a}_i = \sum_{i=1}^N u_i{\bf a}_i\in M. \]
If $k_1 = a-1$, replace $k_1+1$ by $0$ in formula (5.9).  We have $v^{(a-1)}_{i_1}\geq p$, so $v_{i_1}\geq q$.  By (5.9), $w_i=v_i$ for $i\neq i_1$ while $w_{i_1} = v_{i_1} - (q-1)$, so in this case (5.7) implies
\[ \sum_{i=1}^N w_i{\bf a}_i = -(q-1){\bf a}_{i_1} + \sum_{i=1}^N v_i{\bf a}_i\in M \]
by condition (5.5).  Thus in both cases we have $\sum_{i=1}^N w_i{\bf a}_i\in M$.

If $w^{(k_2)}_{i_2}\geq p$ for some $k_2,i_2$, we may repeat the reduction step (5.9).  After finitely many steps we arrive at $\tilde{w}$ with $\tilde{w}^{(k)}_i\leq p-1$ for all $k$ and $i$ and $\sum_{i=1}^N \tilde{w}_i{\bf a}_i\in M$,
hence $\tilde{w}\in U_M$.  But by the construction of $\tilde{w}$ and Eq.~(5.8) we have
\[ w_p(\tilde{w}) = \sum_{i=1}^N \sum_{k=0}^{a-1} \tilde{w}^{(k)}_i<\sum_{i=1}^N\sum_{k=0}^{a-1} v^{(k)}_i\leq w_p(u) = w_p(M). \]
This contradicts the definition of $w_p(M)$, so $\gamma_k$ must be good for all $k$.

It remains to show that $u^{(k)}\in U^+_{\min}(\gamma_k)$ for all $k$.  For each $k$ choose $v^{(k)}\in U^+_{\min}(\gamma_k)$.  Since $\gamma_k$ is good we have $0\leq v^{(k)}_i\leq p-1$ for all $i$ and $k$.  Define
\[ v_i = \sum_{k=0}^{a-1} v^{(k)}_i p^k\quad\text{for $i=1,\dots,N$.} \]
Then $0\leq v_i\leq q-1$ for all $i$ and
\[ \sum_{i=1}^N v_i{\bf a}_i = \sum_{k=0}^{a-1}\bigg(\sum_{i=1}^N v^{(k)}_i p^k\bigg){\bf a}_i = \sum_{k=0}^{a-1} p^k\gamma_k = \sum_{i=1}^N u_i{\bf a}_i, \]
so $v\in U_M$.  Since $v^{(k)}\in U^+_{\min}(\gamma_k)$ for all $k$, we have equality holding in (5.4) for $v$:
\[ w_p(v) = \sum_{k=0}^{a-1} w(\gamma_k). \]
We thus have
\begin{equation}
w_p(v) = \sum_{k=0}^{a-1} w(\gamma_k) \leq \sum_{k=0}^{a-1}\bigg(\sum_{i=1}^N u^{(k)}_i\bigg) = w_p(u).
\end{equation}
Since $u\in U_{M,\min}$, it follows that $v\in U_{M,\min}$ also and that equality holds in~(5.10).  In particular, we must have $w(\gamma_k) = \sum_{i=1}^N u^{(k)}_i$, i.e., $u^{(k)}\in U^+_{\min}(\gamma_k)$, for all $k$.  This completes the proof of Proposition~5.6.
\end{proof}

In summary, we have proved that every $u\in U_{M,\min}$ gives rise (by (5.1) and~(5.2)) to sequences $(u^{(0)},\dots,u^{(a-1)})$ and $(\gamma_0,\dots,\gamma_{a-1})$, with $\gamma_k$ good and $u^{(k)}\in U^+_{\min}(\gamma_k)$ for all $k$, such that (by Eq.~(5.3))
\begin{equation}
\sum_{k=0}^{a-1} p^k\gamma_k\in M.
\end{equation}
Furthermore, equality holds in (5.4), so
\begin{equation}
w_p(M) = \sum_{k=0}^{a-1} w(\gamma_k).
\end{equation}

Conversely, suppose we are given a sequence  $(\gamma_0,\dots,\gamma_{a-1})\in({\mathbb Z}^n)^a$, all $\gamma_k$ good, satisfying (5.11) and (5.12).  Choose $u^{(k)}\in U^+_{\min}(\gamma_k)$ for $k=0,\dots,a-1$ and define $u\in {\mathbb N}^N$ by the formula
\begin{equation}
u_i = \sum_{k=0}^{a-1} u^{(k)}_i p^k \quad\text{for $i=1,\dots,N$.}
\end{equation}
Since all $\gamma_k$ are good we have $0\leq u^{(k)}_i\leq p-1$ for all $k$ and $i$, which implies that $u\in U^+_{q-1}$.  Eq.~(5.11) implies that $u\in U_M$ and Eq.~(5.12) implies that $u\in U_{M,\min}$.

It follows that we can decompose the set $U_{M,\min}$ as follows.  Let $\Gamma_M$ be the set of all sequences $(\gamma_0,\dots,\gamma_{a-1})$ of good elements of ${\mathbb Z}^n$ satisfying (5.11) and (5.12).  For $(\gamma_0,\dots,\gamma_{a-1})\in \Gamma_M$, let $U(\gamma_0,\dots,\gamma_{a-1})$ be the set of all $u$ (necessarily in $U_{M,\min}$ by the previous paragraph) defined by (5.13) with $u^{(k)}\in U^+_{\min}(\gamma_k)$ for $k=0,\dots,a-1$.  We have proved the following result.

\begin{proposition}
There is a decomposition of $U_{M,\min}$ into disjoint subsets:
\[ U_{M,\min} = \bigcup_{(\gamma_0,\dots,\gamma_{a-1})\in\Gamma_M} U(\gamma_0,\dots,\gamma_{a-1}). \]
\end{proposition}

{\bf Example 2:}  Let $A=\{{\bf a}_1,{\bf a}_2\}\subseteq{\mathbb Z}$, where ${\bf a}_1=1$ and ${\bf a}_2 = -1$.  For $\beta\in{\mathbb Z}$ we have $U(\beta) = \{(l+\beta,l)\mid l\in{\mathbb Z}\}$ and $U^+(\beta) = \{ (l+\beta,l)\in U(\beta)\mid l\geq\max(0,-\beta)\}$.  It is then clear that $w(\beta)=|\beta|$ and that
\[ U^+_{\min}(\beta) = \begin{cases} \{(\beta,0)\} & \text{if $\beta\geq 0$,} \\ \{(0,-\beta)\} & \text{if $\beta<0$.}
\end{cases} \]
In particular, $\beta$ is good if and only if $-(p-1)\leq\beta\leq(p-1)$ and no $\beta$ is very good.  Suppose that $q=p$ and $M={\bf e}+(p-1){\mathbb Z}$ where ${\bf e}\in{\mathbb Z}$.  In this case it is easy to compute the decomposition of Proposition~5.14.  We may assume $0\leq {\bf e}<p-1$.  Then 
\[ U_{M,\min} = \\ \begin{cases} U^+_{\min}({\bf e})  & \text{if ${\bf e}<(p-1)/2$,} \\ U^+_{\min}({\bf e}-(p-1))  & \text{if ${\bf e}>(p-1)/2$,} \\ U^+_{\min}((p-1)/2)\cup U^+_{\min}(-(p-1)/2) & \text{if ${\bf e} = (p-1)/2$.} \end{cases} \]

Computing the decomposition of Proposition~5.14 is more involved when one is not working over the prime field.  We continue with this example when $q=p^2$ and $M={\bf e}+(p^2-1){\mathbb Z}$ for ${\bf e}\in{\mathbb Z}$.  
We may assume that ${\bf e} = e_0+e_1p$, where $0\leq e_0,e_1\leq p-1$.  Suppose first that ${\bf e} = 0$.  Then 
\[ U_M = \{(u,u)\mid 0\leq u\leq p^2-1\}\cup \{(p^2-1,0)\}\cup \{(0,p^2-1)\}. \]
It follows that $U_{M,\min} = \{(0,0)\}$ and $w_p(M) = 0$.  

Suppose that ${\bf e}\neq 0$.  If $(\gamma_0,\gamma_1)\in\Gamma_M$, then in particular both $\gamma_0$ and $\gamma_1$ are good, so by the above we have
\begin{equation}
-(p-1)\leq \gamma_0,\gamma_1\leq (p-1).
\end{equation}
By (5.11) we have $\gamma_0+p\gamma_1\in{\bf e}+(p^2-1){\mathbb Z}$, so $(5.15)$ implies that
\begin{equation}
\gamma_0 + p\gamma_1 = {\bf e}
\end{equation}
or
\begin{equation}
\gamma_0 +p\gamma_1 = {\bf e}-(p^2-1).
\end{equation}
Equation (5.16) has a solution $(\gamma_0,\gamma_1) = (e_0,e_1)$ satisfying (5.15).  If $e_0\neq 0$ and $e_1\neq p-1$, then $(\gamma_0,\gamma_1) = (e_0-p,e_1+1)$ is also a solution satisfying (5.15).  There are no other solutions of (5.16) satisfying (5.15).  Equation (5.17) has a solution $(\gamma_0,\gamma_1)=(e_0-(p-1), e_1-(p-1))$ satisfying (5.15).  If $e_0\neq p-1$ and $e_1\neq 0$, then $(\gamma_0,\gamma_1) = (e_0+1,e_1-p)$ is also a solution satisfying (5.15).  There are no other solultions of (5.17) satisfying (5.15).  For notational convenience we set
\[ \Gamma_1 = (e_0,e_1),\;\Gamma_2 = (e_0-p,e_1+1),\;\Gamma_3 = (e_0-(p-1),e_1-(p-1)),\;\Gamma_4 = (e_0+1,e_1-p). \]

Since $w(\gamma_i) = |\gamma_i|$, we have $(\gamma_0,\gamma_1)\in\Gamma_M$ if and only if $w_p(M) = |\gamma_0|+|\gamma_1|$.  From the definition of $w_p(M)$, we have $\Gamma_1\in\Gamma_M$ if and only if
\[ e_0+e_1\leq \min\{p-e_0+e_1+1, 2(p-1)-e_0-e_1,e_0+1+p-e_1\}. \]
Simplifying these inequalities gives
\begin{equation}
\text{$\Gamma_1\in\Gamma_M$ if and only if $e_0,e_1\leq(p+1)/2$ and $e_0+e_1\leq p-1$.}
\end{equation}
Similar calculations give
\begin{equation}
\text{$\Gamma_2\in\Gamma_M$ if and only if $e_0\geq(p+1)/2$ and $e_1\leq(p-3)/2$,}
\end{equation}
\begin{equation}
\text{$\Gamma_3\in\Gamma_M$ if and only if $e_0,e_1\geq(p-3)/2$ and $e_0+e_1\geq p-1$,}
\end{equation}
and
\begin{equation}
\text{$\Gamma_4\in\Gamma_M$ if and only if $e_0\leq(p-3)/2$ and $e_1\geq(p+1)/2$.}
\end{equation}

One then calculates that there are eleven possibilities for $\Gamma_M$: \\
{\bf Case 1}:  $\Gamma_M = \{\Gamma_1\}$ if $e_0\leq(p-3)/2$ and $e_1<(p+1)/2$ or if $e_0=(p-1)/2$ and $e_1<(p-1)/2$. \\
{\bf Case 2}:  $\Gamma_M = \{\Gamma_3\}$ if $e_0\geq(p+1)/2$ and $e_1>(p-3)/2$ or if $e_0=(p-1)/2$ and $e_1>(p-1)/2$. \\
{\bf Case 3}:  $\Gamma_M = \{\Gamma_2\}$ if $e_0>(p+1)/2$ and $e_1<(p-3)/2$. \\ 
{\bf Case 4}:  $\Gamma_M = \{\Gamma_4\}$ if $e_0<(p-3)/2$ and $e_1>(p+1)/2$. \\
{\bf Case 5}:  $\Gamma_M = \{\Gamma_1,\Gamma_2\}$ if $e_0=(p+1)/2$ and $e_1<(p-3)/2$. \\
{\bf Case 6}:  $\Gamma_M = \{\Gamma_3,\Gamma_4\}$ if $e_0=(p-3)/2$ and $e_1>(p+1)/2$). \\
{\bf Case 7}:  $\Gamma_M = \{\Gamma_1,\Gamma_3\}$ if $e_0=e_1=(p-1)/2$. \\
{\bf Case 8}:  $\Gamma_M = \{\Gamma_1,\Gamma_4\}$ if $e_0<(p-3)/2$ and $e_1=(p+1)/2$. \\
{\bf Case 9}:  $\Gamma_M = \{\Gamma_2,\Gamma_3\}$ if $e_0>(p+1)/2$ and $e_1=(p-3)/2$. \\
{\bf Case 10}:  $\Gamma_M = \{\Gamma_1,\Gamma_2,\Gamma_3\}$ if $e_0=(p+1)/2$ and $e_1=(p-3)/2$. \\
{\bf Case 11}:  $\Gamma_M = \{\Gamma_1,\Gamma_3,\Gamma_4\}$ if $e_0=(p-3)/2$ and $e_1 = (p+1)/2$.

\section{Hasse invariants}

We make precise the relationship between the additive character $\Psi$ and the uniformizer $\pi$ of the field ${\mathbb Q}_p(\zeta_{q-1},\zeta_p)$.  Fix $\pi$ satisfying $\pi^{p-1} = -p$.  There is a unique $p$-th root of unity $\zeta_p$ with
$\zeta_p \equiv 1 +\pi \pmod{\pi^2}$.  One defines an additive character $\psi_\pi:{\mathbb F}_p\to{\mathbb Q}_p(\zeta_p)^\times$ by requiring $\psi_\pi(1) = \zeta_p$.  The additive character $\Psi:{\mathbb F}_q\to{\mathbb Q}_p(\zeta_{q-1},\zeta_p)^\times$ is defined to be $\Psi = \psi_\pi\circ{\rm Trace}_{{\mathbb F}_q/{\mathbb F}_p}$.

We recall the method of Ax\cite{A}.  Define a polynomial 
\[ P(t) = \sum_{r=0}^{q-1} p_rt^r \in{\mathbb Q}_p(\zeta_{q-1},\zeta_p)[t] \] 
by the conditions
\[ P(\omega(x)) = \Psi(x)\in{\mathbb Q}_p(\zeta_{q-1},\zeta_p) \]
for $x\in{\mathbb F}_q$ (we take $\omega(0)=0$).  One computes that $p_0=1$, $p_{q-1} = -q/(q-1)$, and that for $1\leq r\leq q-2$ one has $p_r = G_r/(q-1)$, where $G_r$ is the Gauss sum
\[ G_r = \sum_{z\in{\mathbb F}_q^\times} \omega(z)^{-r}\Psi(z) \in{\mathbb Q}_p(\zeta_{q-1},\zeta_p). \]
By Stickelberger's Theorem\cite{St} (or see \cite[Theorem 4.5]{L}),
\begin{equation}
p_r \equiv\frac{\pi^{S(r)}}{r_0!r_1!\cdots r_{a-1}!} \pmod{\pi^{S(r)+1}},
\end{equation}
where $r = r_0 +r_1p + \cdots +r_{a-1}p^{a-1}$ with $0\leq r_k\leq p-1$ for $k=0,\dots,a-1$ and $S(r) = r_0+r_1+\cdots+r_{a-1}$.

For a vector ${\bf b} = (b_1,\dots,b_n)\in{\mathbb Z}^n$ we set $\omega(x)^{\bf b} = \omega(x_1)^{b_1}\cdots \omega(x_n)^{b_n}$.  We also adopt the convention that $\omega(0)^b=0$ for $b\neq 0$ but that $\omega(0)^0 = 1$.  Then for the exponential sum (1.1) we have
\begin{align*}
S(f_\lambda,{\bf e}) &= \sum_{x\in({\mathbb F}_q^\times)^n} \omega(x)^{-{\bf e}} \prod_{j=1}^N \Psi(\lambda_jx^{{\bf a}_j}) \\
 &= \sum_{x\in({\mathbb F}_q^\times)^n} \omega(x)^{-{\bf e}} \prod_{j=1}^N P(\omega(\lambda_j)\omega(x)^{{\bf a}_j}) \\
 &= \sum_{x\in({\mathbb F}_q^\times)^n}  \omega(x)^{-{\bf e}}\prod_{j=1}^N \bigg(\sum_{r=0}^{q-1} p_r\omega(\lambda_j)^r\omega(x)^{r{\bf a}_j}\bigg).
\end{align*}
Let $U^+_{q-1}$ be as in Section 5.  Then
\begin{align*}
S(f_\lambda,{\bf e}) &= \sum_{x\in({\mathbb F}_q^\times)^n}  \omega(x)^{-{\bf e}}\sum_{u=(u_1,\dots,u_N)\in U^+_{q-1}}\bigg(\prod_{j=1}^N p_{u_j}\omega(\lambda_j)^{u_j}\bigg)\omega(x)^{\sum_{j=1}^N u_j{\bf a}_j} \\
 &= \sum_{u\in U^+_{q-1}} \bigg(\prod_{j=1}^N p_{u_j}\omega(\lambda_j)^{u_j}\bigg) \bigg(\sum_{x\in({\mathbb F}_q^\times)^n} \omega(x)^{-{\bf e} + \sum_{j=1}^N u_j{\bf a}_j}\bigg).
\end{align*}
One has
\[ \sum_{x\in({\mathbb F}_q^\times)^n} \omega(x)^{-{\bf e} + \sum_{j=1}^N u_j{\bf a}_j} = 
\begin{cases} (q-1)^n & \text{if $-{\bf e}+\sum_{j=1}^N u_j{\bf a}_j\in(q-1){\mathbb Z}^n$,} \\ 0 & \text{otherwise.} \end{cases} \]
Set $M={\bf e} + (q-1){\mathbb Z}^n$, so that $U_M = \{u\in U^+_{q-1}\mid \sum_{j=1}^N u_j{\bf a}_j\in{\bf e}+(q-1){\mathbb Z}^n\}$.  It follows that
\begin{equation}
S(f_\lambda,{\bf e}) = (q-1)^n\sum_{u\in U_M}\prod_{j=1}^N p_{u_j}\omega(\lambda_j)^{u_j}.
\end{equation}

Note that it can happen that $U_M = \emptyset$ (for example, if ${\bf a}_j\in(q-1){\mathbb Z}^n$ for all $j$ but ${\bf e}\not\in(q-1){\mathbb Z}^n$), in which case $S(f_\lambda,{\bf e})=0$ for all $\lambda$.  From now on we assume that $U_M\neq\emptyset$.  

Eq.~(6.1) implies that
\[ \prod_{j=1}^N p_{u_j}\omega(\lambda_j)^{u_j} \equiv\bigg(\frac{\omega(\lambda_1)^{u_1}\cdots\omega(\lambda_N)^{u_N}}{\prod_{j=1}^N u^{(0)}_j!\cdots u^{(a-1)}_j!}\bigg) \pi^{w_p(u)} \pmod{\pi^{w_p(u)+1}} \]
where the $u^{(k)}$ are defined by (5.1), so (6.2) gives
\begin{multline}
S(f_\lambda,{\bf e}) \equiv \\ 
(-1)^n \bigg(\sum_{u\in U_{M,\min}} \frac{\omega(\lambda_1)^{u_1}\cdots \omega(\lambda_N)^{u_N}}{\prod_{j=1}^N u^{(0)}_j! u^{(1)}_j!\cdots u^{(a-1)}_j!}\bigg)\pi^{w_p(M)}  \pmod{\pi^{w_p(M)+1}}.
\end{multline}
The coefficient of $\pi^{w_p(M)}$ on the right-hand side of (6.3) is a sum of distinct nonzero monomials in $\omega(\lambda_1),\dots,\omega(\lambda_N)$, hence is a nonzero polynomial of degree $\leq q-1$ in each $\omega(\lambda_i)$.

Let $H(\lambda_1,\dots,\lambda_N)$ be the reduction mod $p$ of the coefficient of $\pi^{w_p(M)}$ in Eq.~(6.3):
\begin{multline}
H(\lambda_1,\dots,\lambda_N) = \\ 
(-1)^n\sum_{u\in U_{M,\min}} \frac{\lambda_1^{u_1}\cdots \lambda_N^{u_N}}{\prod_{j=1}^N u^{(0)}_j! u^{(1)}_j!\cdots u^{(a-1)}_j!} \in{\mathbb F}_p[\lambda_1,\dots,\lambda_N].
\end{multline}
By Proposition 5.14 we have
\begin{equation}
H(\lambda_1,\dots,\lambda_N) = (-1)^n\sum_{(\gamma_0,\dots,\gamma_{a-1})\in\Gamma_M} \sum_{u\in U
(\gamma_0,\dots,\gamma_{a-1})}\frac{\prod_{j=1}^N \prod_{k=0}^{a-1} \lambda_j^{u^{(k)}_jp^k}}{\prod_{j=1}^N \prod_{k=0}^{a-1} u^{(k)}_j!}.
\end{equation}
The inner sum on the right-hand side may be written
\[ \prod_{k=0}^{a-1} \sum_{u^{(k)}\in U^+_{\rm min}(\gamma_k)} \frac{(\lambda_1^{u^{(k)}_1}\cdots\lambda_N^{u^{(k)}_N})^{p^k}}{u^{(k)}_1!\cdots u^{(k)}_N!}. \]
But this expression is just $\prod_{k=0}^{a-1} F_{\gamma_k}(\lambda^{p^k})$,  where the $F_\gamma$ are defined in (2.4).  This completes the proof of the following statement, which is our main result. 

\begin{theorem}
The Hasse invariant of the exponential sum $S(f_{\lambda},{\bf e})$ is 
\[ H(\lambda_1,\dots,\lambda_N) = (-1)^n \sum_{(\gamma_0,\dots,\gamma_{a-1})\in \Gamma_M} F_{\gamma_0}(\lambda) F_{\gamma_1}(\lambda^p)\cdots F_{\gamma_{a-1}}(\lambda^{p^{a-1}}). \]
If $H(\lambda_1,\dots,\lambda_N)\neq 0$, then
\[ {\rm ord}_p\;S(f_{\lambda},{\bf e}) = \frac{w_p(M)}{p-1}, \]
and if $H(\lambda_1,\dots,\lambda_N)=0$, then
\[ {\rm ord}_p\;S(f_\lambda,{\bf e})>\frac{w_p(M)}{p-1}. \]
\end{theorem}

{\bf Example 2 (cont.):}  We compute the Hasse invariants of the ``twisted'' Kloosterman sums over ${\mathbb F}_q$ for $q=p,p^2$:
\begin{equation}
\sum_{x\in{\mathbb F}_q^\times} \omega(x)^{-{\bf e}}\Psi(\lambda_1x + \lambda_2/x), 
\end{equation}
which corresponds to choosing $A=\{{\bf a}_1,{\bf a}_2\}\subseteq{\mathbb Z}$, ${\bf a}_1=1$, ${\bf a}_2=-1$.  
First take $q=p$, so that $0\leq{\bf e}<p-1$ and $M={\bf e} + (p-1){\mathbb Z}$.  From Example~2 of the previous section we have immediately
\[ H(\lambda) = \begin{cases} { \displaystyle -\frac{\lambda_1^{\bf e}}{{\bf e}!}} & \text{if ${\bf e}<(p-1)/2$,} \\ {\displaystyle -\frac{\lambda_2^{p-1-{\bf e}}}{(p-1-{\bf e})!}} & \text{if ${\bf e}>(p-1)/2$,} \\ {\displaystyle -\frac{\lambda_1^{(p-1)/2}+\lambda_2^{(p-1)/2}}{((p-1)/2)!}} & \text{if ${\bf e} = (p-1)/2$.}  \end{cases} \]
 
Now consider the case $q=p^2$.  We have $0\leq{\bf e}< p^2-1$ and $M={\bf e}+(p^2-1){\mathbb Z}$, so we may write ${\bf e} = e_0+e_1 p$ with $0\leq e_0,e_1\leq p-1$.  In this situation, we computed in Example 2 of the previous section all possibilities for $\Gamma_M$ with their corresponding elements $(\gamma_0,\gamma_1)$.  Applying Theorem~6.6 gives the following results, where the case numbers refer to the case specified in Example~2 of Section~5. \\
{\bf Case 1:} ${\displaystyle H(\lambda) = -\frac{\lambda_1^{e}}{e_0! e_1!}}$.  \\
{\bf Case 2:} ${\displaystyle H(\lambda) = -\frac{\lambda_2^{p^2-1-e}}{(p-1-e_0)!(p-1-e_1)!}}$. \\
{\bf Case 3:} ${\displaystyle H(\lambda) = -\frac{\lambda_2^{p-e_0}\lambda_1^{(e_1+1)p}}{(p-e_0)!(e_1+1)!}}$. \\
{\bf Case 4:}  ${\displaystyle H(\lambda)= -\frac{\lambda_1^{e_0+1}\lambda_2^{(p-e_1)p}}{(e_0+1)!(p-e_1)!}}$. \\
{\bf Case 5:}  ${\displaystyle H(\lambda) =  -\frac{\lambda_1^{e}}{e_0! e_1!}-\frac{\lambda_2^{p-e_0}\lambda_1^{(e_1+1)p}}{(p-e_0)!(e_1+1)!}}$. \\
{\bf Case 6:}  ${\displaystyle H(\lambda) = -\frac{\lambda_2^{p^2-1-e}}{(p-1-e_0)!(p-1-e_1)!}-\frac{\lambda_1^{e_0+1}\lambda_2^{(p-e_1)p}}{(e_0+1)!(p-e_1)!} }$.  \\
{\bf Case 7:} ${\displaystyle H(\lambda) = -\frac{\lambda_1^{e}}{e_0! e_1!} -\frac{\lambda_2^{p^2-1-e}}{(p-1-e_0)!(p-1-e_1)!}}$. \\
{\bf Case 8:}  ${\displaystyle H(\lambda) = -\frac{\lambda_1^{e}}{e_0! e_1!}-\frac{\lambda_1^{e_0+1}\lambda_2^{(p-e_1)p}}{(e_0+1)!(p-e_1)!}}$.  \\
{\bf Case 9:} ${\displaystyle H(\lambda) = -\frac{\lambda_2^{p-e_0}\lambda_1^{(e_1+1)p}}{(p-e_0)!(e_1+1)!} -\frac{\lambda_2^{p^2-1-e}}{(p-1-e_0)!(p-1-e_1)!}}$. \\
{\bf Case 10:}  ${\displaystyle H(\lambda) = -\frac{\lambda_1^{e}}{e_0! e_1!} -\frac{\lambda_2^{p-e_0}\lambda_1^{(e_1+1)p}}{(p-e_0)!(e_1+1)!}-\frac{\lambda_2^{p^2-1-e}}{(p-1-e_0)!(p-1-e_1)!}}$.  \\
{\bf Case 11:}  ${\displaystyle H(\lambda) = -\frac{\lambda_1^{e}}{e_0! e_1!} -\frac{\lambda_2^{p^2-1-e}}{(p-1-e_0)!(p-1-e_1)!}-\frac{\lambda_1^{e_0+1}\lambda_2^{(p-e_1)p}}{(e_0+1)!(p-e_1)!}}$. 

{\bf Remark:} In \cite{AS} we calculated $w_p(M)$ for the exponential sum (6.7) for all odd~$q$ and all ${\bf e}$.  From the results in that paper, one can also compute the Hasse invariants for those exponential sums.

\section{Affine sums}

In this section we extend Theorem 6.6 to the case of exponential sums containing some affine variables.  Let now
\[ f_\lambda = \sum_{j=1}^N \lambda_jx^{{\bf a}_j}\in{\mathbb F}_q[x_1^{\pm 1},\dots,x_m^{\pm 1}, x_{m+1},\dots,x_n], \]
i.e., the first $m$ variables are toric and the last $n-m$ variables are affine.  We are considering the sum
\begin{equation}
S_{\rm aff}(f_\lambda,{\bf e}) = \sum_{x\in({\mathbb F}_q^\times)^m\times{\mathbb F}_q^{n-m}} \prod_{i=1}^m \omega(x_i)^{-e_i} \Psi(f_\lambda(x))\in{\mathbb Q}_p(\zeta_p,\zeta_{q-1}),
\end{equation}
where ${\bf e} = (e_1,\dots,e_m,0,\dots,0)\in{\mathbb Z}^n$.  For each subset $I\subseteq\{m+1,\dots,n\}$, let $f_{\lambda,I}\in{\mathbb F}_q[x_1^{\pm 1},\dots,x_m^{\pm 1},\{x_i\}_{i\in I}]$ be the Laurent polynomial obtained from $f_{\lambda}$ by setting $x_j=0$ for $j\not\in I$, $m+1\leq j\leq n$, and let ${\bf e}_I = (e_1,\dots,e_m,0,\dots,0)\in{\mathbb Z}^{m+|I|}$.  
As in (1.1), we set
\[ S(f_{\lambda,I},{\bf e}_I) = \sum_{x\in({\mathbb F}_q^\times)^{m+|I|}} \prod_{i=1}^m \omega(x_i)^{-e_i} \Psi(f_{\lambda,I}(x)). \]
Then we have the relation
\begin{equation}
S_{\rm aff}(f_{\lambda},{\bf e}) = \sum_{I\subseteq\{m+1,\dots,n\}} S(f_{\lambda,I},{\bf e}_I).
\end{equation}

We identify ${\mathbb Z}^{m+|I|}$ with the subgroup of ${\mathbb Z}^n$ of vectors with $j$-th coordinate $0$ for $j\not\in I$.  Let $M_I = {\bf e}_I + (q-1){\mathbb Z}^{m+|I|}$ and let $A_I$ be the subset of $A$ consisting of all ${\bf a}_i$ with $j$-th coordinate $0$ for $j\not\in I$.  Then we have  
\[ U_{M_I} = \{u\in U_M\mid  u_i= 0\text{ for every index $i$ for which }{\bf a}_i\not\in A_I\}. \]
From (6.2) we have
\begin{equation}
S(f_{\lambda,I},{\bf e}_I) = (q-1)^{m+|I|}\sum_{u\in U_{M_I}} \prod_{j=1}^N p_{u_j}\omega(\lambda_j)^{u_j}.
\end{equation}
Combining (7.2) and (7.3) gives
\begin{equation}
S_{\rm aff}(f_\lambda,{\bf e}) = \sum_{I\subseteq\{m+1,\dots,n\}} (q-1)^{m+|I|}\sum_{u\in U_{M_I}} \prod_{j=1}^N p_{u_j}\omega(\lambda_j)^{u_j}.
\end{equation}

The sets $U_{M_I}$ are not disjoint so we gather the contributions on the right-hand side of (7.4) as follows.  Fix $u\in U_M$
and let $I_u\subseteq\{m+1,\dots,n\}$ be the set of indices~$j$ such that the $j$-th coordinate of $\sum_{i=1}^N u_i{\bf a}_i$ is nonzero.  Then we have $u\in M_I$ if and only if $I_u\subseteq I$, and the contribution of $u$ to the sum (7.4) is
\[ \prod_{j=1}^N p_{u_j}\omega(\lambda_j)^{u_j}\sum_{I\supseteq I_u} (q-1)^{m+|I|}  = (q-1)^{m+
|I_u|}q^{n-m-|I_u|} \prod_{j=1}^N p_{u_j}\omega(\lambda_j)^{u_j}. \]
For $l=0,1,\dots,n-m$, let $M_l\subseteq M$ be the subset of lattice points with $j$-th coordinate nonzero for exactly $l$ values of $j\in\{m+1,\dots,n\}$.  Then $u\in U_{M_l}$ if and only if $|I_u|=l$, so
\begin{equation}
S_{\rm aff}(f_\lambda,{\bf e}) = \sum_{l=0}^{n-m} (q-1)^{m+l} q^{n-m-l} \sum_{u\in U_{M_l}} \prod_{j=1}^N p_{u_j}\omega(\lambda_j)^{u_j}.
\end{equation}
Note that the sets $M_l$ satisfy condition (5.5), since for ${\bf a} = \sum_{i=1}^N c_i{\bf a}_i\in{\mathbb N}A\cap M$ the value of $l$ for which ${\bf a}\in M_l$ is determined by the set $\{i\in \{1,\dots,N\}\mid c_i>0\}$, and this set is unchanged if some $c_{i_0}\geq q$ is replaced by $c_{i_0}-(q-1)$.

We now impose the hypothesis that for each $j\in \{m+1,\dots,n\}$ there exists ${\bf a}_i\in A$ such that the $j$-th coordinate of ${\bf a}_i$ is nonzero, i.e., the set $A$ is not contained in any coordinate hyperplane $x_j=0$, $m+1\leq j\leq n$.
\begin{lemma}
Assume the above hypothesis.  Then for $l=0,1,\dots,n-m$, 
\[ w_p(M_l) + a(n-m-l)(p-1)\geq w_p(M_{n-m}). \]
\end{lemma}

\begin{proof}
The assertion is trivial for $l=n-m$ so suppose that $u\in U_{M_l}$ for some $l<n-m$.  Then there exists an index $j$, $m+1\leq j\leq n$, such that $u_i = 0$ for all~$i$ such that ${\bf a}_i$ has nonzero $j$-th coordinate. By our hypothesis, there exists $i_0$ such that the $j$-th coordinate of ${\bf a}_{i_0}$ is nonzero.  Define
\[ u'_i = \begin{cases} u_i & \text{if $i\neq i_0$,} \\ q-1 & \text{if $i=i_0$.} \end{cases} \]
Then $I_{u'}\supseteq I_u\cup\{j\}$ so $u'\in U_{M_{l'}}$ for some $l'\geq l+1$.  Since $u_{i_0} = 0$, it follows that
\[ w_p(u') = w_p(u) + a(p-1). \]
If $l'<n-m$ this argument may be repeated until after, say, $r$ steps (with $r\leq n-m-l$) we arrive at $u^{(r)}\in U_{M_{n-m}}$, hence
\[ w_p(u^{(r)}) = w_p(u) + ar(p-1)\leq w_p(u) + a(n-m-l)(p-1). \]
This inequality implies the lemma.
\end{proof}

Equation (6.1) implies that
\[ {\rm ord}_p\;\prod_{j=1}^N p_{u_j}\omega(\lambda_i)^{u_j} = \frac{w_p(u)}{p-1}, \]
so
\begin{equation}
{\rm ord}_p\;q^{n-m-l}\prod_{j=1}^N p_{u_j}\omega(\lambda_i)^{u_j} = \frac{w_p(u) + (n-m-l)a(p-1)}{p-1}.
\end{equation}
Then Equations (7.5), (7.7), and Lemma 7.6 imply that
\begin{equation}
S_{\rm aff}(f_\lambda,{\bf e})\equiv {\sum_l}' (-1)^{m+l}\sum_{u\in U_{M_l,\min}} \prod_{j=1}^N p_{u_j}\omega(\lambda_j)^{u_j} \pmod{\pi^{w_p(M_{n-m})+1}}, 
\end{equation}
where $\sum_l'$ denotes a sum over those $l$ for which equality holds in Lemma 7.6.

We now apply Proposition 5.14 to decompose $U_{M_l,\min}$:  Let $\Gamma_{M_l}$ be the set of all sequences $(\gamma_0,\dots,\gamma_{a-1})$ of good elements of ${\mathbb Z}^n$ satisfying (5.11) and (5.12) (with $M$ replaced $M_l$).  Then
\[ U_{M_l,\min} = \bigcup_{(\gamma_0,\dots,\gamma_{a-1})\in\Gamma_{M_l}} U(\gamma_0,\dots,\gamma_{a-1}). \]
Arguing as in the proof of Theorem 6.6 gives the following result.

\begin{theorem}
The Hasse invariant of the exponential sum $S_{\rm aff}(f_\lambda,{\bf e})$ is
\[ H(\lambda_1,\dots,\lambda_N) = {\sum_l}' (-1)^{m+l} \sum_{(\gamma_0,\dots,\gamma_{a-1})\in\Gamma_{M_l}} F_{\gamma_0}(\lambda)F_{\gamma_1}(\lambda^p)\cdots F_{\gamma_{a-1}}(\lambda^{p^{a-1}}), \]
where $\sum_l'$ denotes a sum over those $l$ for which equality holds in Lemma $7.6$.
If $H(\lambda_1,\dots,\lambda_N)\neq 0$, then
\[ {\rm ord}_p\;S_{\rm aff}(f_\lambda,{\bf e}) = \frac{w_p(M_{n-m})}{p-1}, \]
and if $H(\lambda_1,\dots,\lambda_N) = 0$, then
\[ {\rm ord}_p\;S_{\rm aff}(f_\lambda,{\bf e}) > \frac{w_p(M_{n-m})}{p-1}. \]
\end{theorem}

{\bf Example 3:}  Suppose that ${\bf a}_j\in{\mathbb N}^n$ with $\sum_{i=1}^n a_{ji}=d>0$ for all $j$, i.e., $f_\lambda$ is a homogeneous polynomial of degree $d$.  We work over ${\mathbb F}_p$ and consider the sum
\[ S_{\rm aff}(x_{n+1}f_\lambda) = \sum_{(x_1,\dots,x_{n+1})\in{\mathbb F}_p^{n+1}} \Psi(x_{n+1}f_\lambda(x_1,\dots,x_n)). \]
If $N_{\rm aff}(\lambda)$ denotes the number of ${\mathbb F}_p$-rational points on the hypersurface $f_\lambda=0$ in~${\mathbb A}^n$, then $S_{\rm aff}(x_{n+1}f_\lambda) = pN_{\rm aff}(\lambda)$.  We assume there exist $u_i$, $0\leq u_i\leq p-1$, such that $\sum_{i=1}^N u_i({\bf a}_i,1)\in(p-1)({\mathbb Z}_{ >0})^{n+1}$ and $\sum_{i=1}^N u_i=p-1$.  This implies that $w_p(M_{n+1}) = p-1$ (so that $N_{\rm aff}(\lambda)$ is prime to $p$ for generic $\lambda$) and that equality holds in Lemma~7.6 only for $l=n+1$.  Theorem~7.9 then gives
\[ H(\lambda_1,\dots,\lambda_N) = (-1)^n \sum_{\gamma_0\in\Gamma_{M_{n+1}}} F_{\gamma_0}(\lambda_1,\dots,\lambda_N) \]
and $H(\lambda)\equiv N_{\rm aff}(\lambda)\pmod{p}$.  Let $\Gamma_{M_{n+1}} = \{\gamma^{(1)}_0,\dots,\gamma_0^{(r)}\}$.  Then 
\[ F_{\gamma_0^{(j)}}(\lambda_1,\dots,\lambda_N) = \sum_{\sum_i u_i({\bf a}_i,1) = \gamma_0^{(j)}} \frac{\lambda_1^{u_1}\cdots\lambda_N^{u_N}}{u_1!\cdots u_N!}, \]
where the sum is over all $N$-tuples $(u_1,\dots,u_N)$, $0\leq u_i\leq p-1$ for all $i$, such that $\sum_{i=1}^N u_i({\bf a}_i,1) = \gamma_0^{(j)}\in(p-1)({\mathbb Z}_{>0})^{n+1}$ and $\sum_{i=1}^N u_i=p-1$.  Note that $(p-1)!F_{\gamma_0^{(j)}}(\lambda)$ is the coefficient of $x^{\gamma_0^{(j)}}$ in $(x_{n+1}f_\lambda(x))^{p-1}$.
It now follows from Katz\cite[Algorithm~2.3.7.14]{K} that $(-1)^{n+1}H(\lambda)$ is congruent mod $p$ to the trace of the Hasse-Witt matrix associated to the projective hypersurface with equation $f_\lambda=0$.

\end{document}